\documentclass[11pt]{article}
\usepackage{amsmath, amssymb, amsfonts, amsthm}
\usepackage{enumitem}
\usepackage{multicol}
\usepackage{hyperref}
\usepackage{multirow}
\allowdisplaybreaks
\def \N {{\mathbb{N}}}
\def \Z {{\mathbb{Z}}}

\def \al {{\alpha}}
\def \be {{\beta}}
\def \ga {{\gamma}}

\newtheorem*{theorem*}{Theorem}
\newtheorem{theorem}{Theorem}

\newtheorem{lemma}[theorem]{Lemma}
\newtheorem{ex}[theorem]{Example}

\newtheorem*{ex*}{Example}

\title{Permanents of $3\times3$ Invertible Matrices Modulo $n$}
\author{Ayush Bohra and A. Satyanarayana Reddy \\
Department of 
Mathematics, Shiv Nadar 
University, India-201314\\ (e-mail: ab424@snu.edu.in, satya.a@snu.edu.in).

}

\date{}

\begin{document}
\maketitle
\begin{abstract}
 We count the number of elements in the set $$G_{3}(n,x) = \{M \in GL_3(\Z_n) \mid perm(M) \equiv x \pmod{n} \}.$$
\end{abstract}

\section{Introduction}\label{sec:intro}
In~\cite{A:S}, we studied the action of the permanent on $GL_2(\Z_n)$ and counted the number of invertible matrices which have permanent congruent to $x \pmod{n}.$ Let $$G_{3}(n,x) = \{M \in GL_3(\Z_n) \mid perm(M) \equiv x \pmod{n} \}.$$ Let $g_3(n,x)$ denote  the cardinality of $G_{3}(n,x).$ We recall some results from the discussion of the $2\times2$ case which carry over naturally to $3\times3$ matrices as well. \begin{enumerate}
    \item \textbf{Multiplicative Property (MP):} For $a,b,x \in \N , (a,b) = 1 , $ we have $g_{3}(ab,x) = g_{3}(a,x)\times g_{3}(b,x).$
    \item \textbf{Invariance Property (IP):} Let $p$ be a prime number and $k \in \N.$ Let $m \in \N$ be such that $p \nmid m.$ Then for $1 \leq r \leq k$ we have $g_{3}(p^k,p^r) = g_{3}(p^k,mp^r).$ We also have $g(p^k,1) = g(p^k,u)$ whenever $p \nmid u.$ This property actually partitions $GL_{3}(\Z_{p^{k}})$ in the following manner : \begin{align*}
    |GL_{3}(\Z_{p^{k}})| &= \sum\limits_{i=0}^k \varphi(p^i)\times g(p^k,p^{k-i}).
    \end{align*}
\end{enumerate} 
Since in this paper, we are dealing primarily with $3\times3$ matrices, from now on we will simply write $g(n,x)$ to denote $g_{3}(n,x).$ When we will talk specifically about $2\times2$ matrices, we will use the notation $g_{2}(n,x).$

We now elaborate  the organization of the paper. Our main goal in this paper is to compute $g(n,x)$. Since this function is multiplicative in $n,$ it is sufficient to know $g(p^k,x),$ where $p$ is a prime number and $k \in \N.$ Therefore the main result we will establish is the following.
\begin{theorem}\label{thm:th1}
Let $p$ be a prime. Let $k,x \in \N.$ Then
$$g(p^k,x) = \begin{cases} p^{8(k-1)}g(p,0) & \mbox{if } p\mid x, \\ \frac{p^{8(k-1)[|GL_3(\Z_p)| - g(p,0)]}}{p-1} & \mbox{otherwise.}
    \end{cases}$$
\end{theorem}

These values look arbitrary at first but they are an extension of the analogous result in \cite{A:S}.
\begin{theorem} \label{thm:th2}
Let $p$ be a prime. Let $k,x \in \N.$ Then $g(p^k,x)$ assumes only two distinct values. Specifically,
    $$g(p^k,x) = \begin{cases} g(p^k,0) & \mbox{if } p\mid x, \\ g(p^k,1) & \mbox{otherwise.}
    \end{cases}$$
\end{theorem}

This looks similar to the IP but it has much more. For example,  from IP we cannot get $g(p^4,p) = g(p^4,2p^2),$ where $p$ is an odd prime number but one can by using Theorem~\ref{thm:th2}. Further knowing $g(p^k,0)$ is sufficient to determine $g(p^k,1).$  Since $\forall i, 1\leq i \leq k$ we have $g(p^k,0) = g(p^k,p^i)$, and the telescoping sum :
\begin{align*}
 |GL_{3}(\Z_{p^{k}})| &=  g(p^k,0) + \sum\limits_{i=1}^{k-1} (p^i - p^{i-1})\times g(p^k,0) + \varphi(p^k)\times g(p^k,1) \\ 
 &=  p^{k-1}\times g(p^k,0) + (p^k - p^{k-1})\times g(p^k,1).
\end{align*}

Since $|GL_{3}(\Z_{p^{k}})| = p^{9(k-1)}|GL_{3}(\Z_{p})| = p^{9(k-1)}(p^3-1)(p^3-p)(p^3-p^2),$ once we know $g(p^k,0)$ we can also evaluate $g(p^k,1).$ To do that, we will establish the following relation : 
\begin{theorem}\label{thm:th3}
Let $p$ be a prime and $k\in \N.$ Then   $g(p^k,0) = p^{8(k-1)}g(p,0).$
\end{theorem}
Thus the paper boils down to computing $g(p,0).$ It is clear that that $g(2,0) = 0.$ The following result will be proved in the last section~\ref{sec:g(p,0)}.

\begin{theorem} \label{thm:th4}
Let $p$ be an odd prime. Then 
    $$g(p,0) = \begin{cases} p(p-1)^4[(p+1)^3 + 1] & \mbox{if } (p-3) \  \mbox{is a quadratic residue modulo  $p,$} \\ p^2(p-1)^4(p^2 + 3p + 5) & \mbox{otherwise.}\\
    \end{cases}$$
\end{theorem}
We will dedicate the Section~\ref{sec:range} to the proof of Theorem~\ref{thm:th2}. In Section~\ref{sec:computation} we derive the proof of Theorem~\ref{thm:th3} using Theorem~\ref{thm:th2}. Theorem~\ref{thm:th1} will  follow immediately. In  Section~\ref{sec:g(p,0)} we prove Theorem~\ref{thm:th4} in two different ways.

\section{The range of $g(p^k,x)$}\label{sec:range}
In this section we  prove Theorem~\ref{thm:th2}. We will start by showing that $g(p^k,0) = g(p^k,p^r)$ whenever $1\leq r\leq k.$ That coupled with IP will prove Theorem~\ref{thm:th2}. For this section, we introduce some notation. Let $P_{ij}$ be  the permanent of  $2\times 2$ submatrix obtained by deleting row $i$ and column $j.$ Given a prime $p$, $k, x \in \N,$ $1 \leq i,j \leq 3,$ define the following sets :
\begin{itemize}

    \item $G(p^k,x,1,1) = \{A \in G(p^k,x) \mid p\nmid P_{11} \}$
    \item $G(p^k,x,1,2) = \{A \in G(p^k,x) \mid p\mid P_{11}, p\nmid P_{12} \}$ 
    \item $G(p^k,x,1,3) = \{A \in G(p^k,x) \mid p\mid P_{11}, P_{12} , p \nmid P_{13} \}$
    \item $G(p^k,x,2,1) = \{A \in G(p^k,x) \mid p\mid P_{11}, P_{12}, P_{13}, p \nmid P_{21} \}$
    \item $G(p^k,x,2,2) = \{A \in G(p^k,x) \mid p\mid P_{11}, P_{12}, P_{13}, P_{21}, p \nmid P_{22} \}$
    \end{itemize}
In a similar way, one can define $G(p^k,x,2,3), G(p^k,x,3,1), G(p^k,x,3,2)$ and $G(p^k,x,3,3).$

An interesting observation is that for an odd prime $p$ and $p\mid x,$ the sets $$G(p^k,x,1,1), G(p^k,x,1,2), \\ G(p^k,x,1,3), G(p^k,x,2,1) \;\mbox{and} \;G(p^k,x,2,2)$$ are always non-empty, containing the matrices 
$$\begin{pmatrix}
\frac{x -1}{2} & \frac{x + 1}{2} & 0 \\
1 & 1 & 0 \\
0 & 0 & 1
\end{pmatrix} ,
\begin{pmatrix}
1 & 0 & 0\\
1 & 1 & 1 \\
0 & 1 & x - 1
\end{pmatrix}, \begin{pmatrix}
1 & 0 & 0\\
(x - 1)^{-1} & 1 & 1 \\
x -1 & 1 & x - 1
\end{pmatrix}, \begin{pmatrix}
0 & 0 & 1\\
1 & 1 & 0 \\
x - 1 & 1 & 0
\end{pmatrix}, \begin{pmatrix}
1 & 1 & 1\\
0 & 1 & 1 \\
0 & 1 & x - 1
\end{pmatrix}$$ respectively. Can we assure the non-emptiness of the remaining four sets, namely $$G(p^k,x,2,3),  G(p^k,x,3,1), G(p^k,x,3,2)\;\mbox{and}\;G(p^k,x,3,3)?$$ The following lemma shows that it is not possible, even when $x$ is not divisible by $p.$
\begin{lemma}\label{lem:Gpkoij}
Let $p$ be a prime number. Let $k, x \in \N.$  Then the sets $$G(p^k,x,2,3), G(p^k,x,3,1), G(p^k,x,3,2) \;\mbox{and}\; \; G(p^k,x,3,3)$$ are empty sets. 
\end{lemma}
\begin{proof}
We will show that if  $A=[a_{ij}]$ with $p$ dividing $P_{11}, P_{12}, P_{13}, P_{21}$ and $P_{22}$ then $A$ will fail to be invertible.   It is easy to verify the following identity.
    $$2a_{22}P_{22} - a_{11}P_{11} + a_{12}P_{12} - 2a_{21}P_{21} - 3a_{13}P_{13} = \det(A) - 6a_{13}a_{21}a_{32}.$$

At this stage, the lemma is already proved for $p=2$ and $p=3.$  We give the proof for the remaining primes.  Suppose that $p$ had divided any one of $a_{13}, a_{21} $ or $a_{32}.$ Then we would be done.

We again assume the contrary, {\it i.e.,} suppose ${a_{13}}^{-1}, {a_{21}}^{-1} $ and ${a_{32}}^{-1}$ exist. Since $$p|P_{11} = a_{22}a_{33} + a_{23}a_{32},\;\;\;a_{23} \equiv a_{22}a_{33}a_{32}^{-1}\pmod p.$$ We substitute this into $P_{12} = a_{21}a_{33} + a_{23}a_{31}\equiv 0 \pmod p$  and we get  $ a_{33}(a_{21}a_{32} - a_{22}a_{31}) \equiv 0 \pmod{p}.$ 
Suppose $p \mid a_{33}.$ Then from $p|P_{11},$ we see that $p \mid a_{23}$ and from $p|P_{21},$ we get $p \mid a_{13}.$ Therefore $p$ divides the third column and thus $A$ cannot be invertible. Thus the only option is that $p \mid (a_{21}a_{32} - a_{22}a_{31}).$ With a similar argument we can show that  $p \mid a_{31}(a_{23}a_{32} - a_{22}a_{33}).$ Finally we are left with \begin{center}
    $p \mid (a_{21}a_{32} - a_{22}a_{31})$ and $p \mid (a_{22}a_{33} - a_{23}a_{32}).$
\end{center}
Multiplying the first equation by $a_{33}$ and the second by $a_{31}$ and subtracting the second from the second we arrive at $p \mid a_{32}(a_{21}a_{33} - a_{23}a_{31}).$ Since $p\nmid a_{32},$ we have $p \mid (a_{21}a_{33} - a_{23}a_{31})$ and hence $p \mid \det(A), $ a contradiction. 
\end{proof}

Theorem~\ref{thm:th5} in Section~\ref{sec:g(p,0)} provides exact values of $g(p,0,i,j).$ 

\begin{lemma}\label{lem:0x}
Let $p$ be a prime number. Let $k, x \in \N$ be such that $p\mid x.$ Then $g(p^k,0) = g(p^k,x).$
\end{lemma}
\begin{proof}
Instead of giving a single bijection from $G(p^k,0)$ to $G(p^k,x)$ we will give a bijection from $G(p^k,0,i,j)$ to $G(p^k,x,i,j).$ Since we will do this $\forall \ 1\leq i,j \leq 3,$ we would have in effect shown $g(p^k,0) = g(p^k,x) $ whenever $p\mid x.$

After having the Lemma~\ref{lem:Gpkoij}, we only need to prove the above for $$G(p^k,x,1,1), G(p^k,x,1,2), G(p^k,x,1,3), G(p^k,x,2,1)\;\mbox{and}\; G(p^k,x,2,2).$$ We prove it for $G(p^k,0,1,1).$ The remaining four can then be proved in a similar way. Define the map : \begin{center}
    $\psi_{11} : G(p^k,0,1,1) \to G(p^k,x,1,1) $ by $\begin{pmatrix}
    a_{11} & a_{12} & a_{13} \\
    a_{21} & a_{22} & a_{23} \\
    a_{31} & a_{32} & a_{33}
    \end{pmatrix} \mapsto \begin{pmatrix}
    a_{11} + \frac{x}{P_{11}}& a_{12} & a_{13} \\
    a_{21} & a_{22} & a_{23} \\
    a_{31} & a_{32} & a_{33}
    \end{pmatrix}.$
\end{center}
The important thing to see here is that the map $\psi_{11}$ preserves invertibility since we are translating the determinant by a multiple of $p.$ This would not be guaranteed if we had done a similar operation from $G(p^k,1,1,1)$ to $G(p^k,x,1,1),\ p\mid x.$
Now it is easy to check that $\psi_{11}$ is an injective map. For surjectivity, if $M =  \begin{pmatrix}
    b_{11} & b_{12} & b_{13} \\
    b_{21} & b_{22} & b_{23} \\
    b_{31} & b_{32} & b_{33}
    \end{pmatrix} \in G(p^k,x,1,1)$, then $\psi_{11}(N) = M,$ where

    $N = \begin{pmatrix}
    b_{11} - \frac{x}{P_{11}} & b_{12} & b_{13} \\
    b_{21} & b_{22} & b_{23} \\
    b_{31} & b_{32} & b_{33}
    \end{pmatrix} \in G(p^k,0,1,1).$ Consequently  the map $\psi_{11}$ is a bijection. Hence we have $|G(p^k,0,1,1)| = |G(p^k,x,1,1)|.$ 
In a similar way we can show 
$$|G(p^k,0,1,2)| = |G(p^k,x,1,2)|, |G(p^k,0,1,3)| = |G(p^k,x,1,3)|,$$ 
 $$|G(p^k,0,2,1)| = |G(p^k,x,2,1)| \; \mbox{and} \; |G(p^k,0,2,2)| = |G(p^k,x,2,2)|.$$ But since their disjoint unions are exactly $G(p^k,0)$ and $G(p^k,x)$  respectively,  
 thus  we have    $$g(p^k,0) = g(p^k, x), \;\;\mbox{whenever}\; p\mid x.$$
\end{proof}

Now the  proof of Theorem~\ref{thm:th2} can be obtained  from IP which gives $g(p^k,1) = g(p^k,y)$ when $p\nmid y$ and Lemma~\ref{lem:0x} which gives $g(p^k,0) = g(p^k,x)$ when $p\mid x.$

\section{Computing $g(p^k,x)$}\label{sec:computation}
In this section, our goal is to prove Theorem~\ref{thm:th3}.
In particular, from Theorem~\ref{thm:th2} it is sufficient  to find $g(p^k,0)$ and $g(p^k,1)$ and furthermore,  with our earlier discussion, it reduces in  finding $g(p^k,0).$ For a given a given odd prime $p$ and $k\in \N,$ we define the set $G_{p^k}= \bigcup\limits_{p\mid x} G(p^k,x).$ We have $|G_{p^k}|=p^{k-1}g(p^k,0),$ as $|\{x|\;1\le x\le n,p|x\}|= p^k - \varphi(p^k)=p^{k-1}.$

We define the map $\pi_{p^k}: G_{p^k}\to G(p,0)$ as $[c_{ij}] \mapsto [c_{ij} \pmod{p}].$ It is easy to see that if $A=[a_{ij}]\in G(p,0),$ then $|\pi^{-1}_{p^k}(A)|=p^{9(k-1)}$ as if $B=[b_{ij}]\in \pi^{-1}_{p^k}(A),$ then 
$b_{ij}=a_{ij}+c_{ij},$ where $p|c_{ij}.$ 
The following example illustrates the same.

\begin{ex}
Consider the case when $p=3$ and $k=2.$ There, the set $G_9$ will be as follows : 
    $$G_9 = G(9,0)\bigcup G(9,3) \bigcup G(9,6).$$
The map $\pi_9$ reduces every matrix in $G$ to a matrix in $G(3,0).$ Furthermore, each element of $G(3,0)$ has $3^9 = 19683$ pre-images. For example consider the matrix $\begin{pmatrix}
1 & 0 & 0 \\
2 & 1& 2\\
1 & 1& 1
\end{pmatrix}\in G(3,0).$ It has a preimage $\begin{pmatrix}
4 & 0 & 0 \\
2 & 1& 2\\
1 & 1& 1
\end{pmatrix} \in G(9,3).$ The matrix $\begin{pmatrix}
1 & 3 & 0 \\
2 & 1& 2\\
1 & 1& 1
\end{pmatrix} \in G(9,6)$ is also an example of a preimage. 
\end{ex}

Now the proof of Theorem~\ref{thm:th3} follows from  the following equation which derived from discussion before the previous example. 
\begin{align*}
    p^{k-1}g(p^k,0) &= p^{9(k-1)}g(p,0), \\ 
    g(p^k,0) &= p^{8(k-1)}g(p,0).
\end{align*}

 Now we can also get exact value of $g(p^k,1)$ in the following manner.  
\begin{align*}
    |GL_3(\Z_{p^k})| &= p^{9(k-1)}\times|GL_3(\Z_p)|\\
    &= p^{(k-1)}\times g(p^k,0) + (p^k - p^{(k-1)})\times g(p^k,1). \\ 
    g(p^k,1) &= \frac{p^{8(k-1)[|GL_3(\Z_p)| - g(p,0)]}}{p-1}.
\end{align*}

\section{Computation of  $g(p,0)$}\label{sec:g(p,0)}
The following table summarizes how we compute $g(p,0).$ We illustrate few cases, the procedure for the remaining cases is similar. \\
\vglue 2mm
\begin{tabular}{|l|l|l|}
\hline
 Condition & Subconditions & Number of matrices\\
 \hline
      \multirow{3}{*}{Only one entry in the 1st row is nonzero}&& \\
      && $3p^2(p-1)^4$\\ 
      &&\\
      \hline
\multirow{3}{*}{Only one entry in 1st row is zero} &Only one entry in 2nd row is nonzero& $3p(p-1)^4$\\ 
      &Only one entry in  2nd row is zero& $6p(p-1)^5 +3p(p-1)^4$ \\
      &All the entries of 2nd row are nonzero&$3p(p-1)^6$\\
      \hline
\multirow{3}{*}{All the entries of 1st row are nonzero} & Only one entry in 2nd row is nonzero&$3p(p-1)^5$ \\ 
      &Only one entry in  2nd row is zero & $3p(p-1)^6$ \\
      &All the entries of 2nd row are nonzero& **\\
      \hline
  \multicolumn{3}{|l|}
  {Here $**=\begin{cases}
                    p^2(p-1)^5(p-2) & \mbox{if $p-3$ is  not a quadratic residue modulo $p,$}\\
                   p(p-1)^5(p^2-2p-2) & \mbox{if $p-3$ is a quadratic residue modulo $p.$}\\
                   \end{cases}$
        }
 \\

      \hline
    \end{tabular}

\begin{itemize}
 \item When the first row contains only one non-zero entry:  Without loss of generality, let $A=[a_{ij}]$ with $a_{11} \neq 0, a_{12} = a_{13} = 0.$ Clearly, $a_{11}$ has a total of $(p-1)$ choices. We now look at the possible choices for the other six entries of $A.$ Consider the submatrix $\begin{pmatrix}
a_{22} & a_{23} \\
a_{32} & a_{33}
\end{pmatrix}.$ The permanent of this matrix has to be divisible by $p$ and determinant cannot be divisible by $p.$  Therefore there are a total of $g_2(p,0) = (p-1)^3$ \cite{A:S} choices for the four entries $a_{22}, a_{23}, a_{32}$ and $a_{33}.$ The remaining two entries $a_{21}$ and $a_{31}$ each have $p$ choices. So the total number of ways becomes $p^2 \times (p-1)^4.$ Hence  total choices in this case is $3p^2 \times (p-1)^4.$

\item Exactly two entries in the first row are nonzero and all the entries in the second row are nonzero:
Let $A = \begin{pmatrix}
            a_{11} & a_{12} & 0 \\
            a_{21} & a_{22} & a_{23} \\
            x & y & z 
            \end{pmatrix} \in G(p,0).$ Then each entry in the second row
has $(p-1)$ choices each. The congruences for A are now \begin{align*}
    perm(A)&= a_{11}(a_{22}z + a_{33}y) + a_{12}(a_{21}z + a_{23}x) \equiv 0 \pmod{p}, \\ 
    \det(A)&= a_{11}(a_{22}z - a_{33}y) - a_{12}(a_{21}z - a_{23}x) \not\equiv 0 \pmod{p}.
\end{align*}
Let $z$ be the free variable, so it has $p$ choices. Now we will exclude certain values of $x$ and $y$ which will fail invertibility of $A.$ We do not want the following system to have a solution, after we fix $z = z_{o}$ : 
\begin{align*}
    a_{12}a_{23}x + a_{11}a_{23}y & \equiv -z_{o}(a_{11}a_{22} + a_{12}a_{21}), \\ 
    a_{12}a_{23}x - a_{11}a_{23}y &\equiv -z_{o}(a_{11}a_{22} - a_{12}a_{21}).
\end{align*} 
Since the determinant of the coefficient matrix is invertible, this system admits a unique solution, which will make $A$ non-invertible. So we exclude that one choice of $(x,y).$ Now choosing any of the $(p-1)$ values of $x$ and using the permanent congruence to get a uniquely determined value of $y$ will work. Thus there are a total of $(p-1)$ ways to choose the ordered pair $(x,y).$ After considering choices for the first row we end up  with $3p(p-1)^6$ ways.

\item If all the entries of first row are nonzero and only one entry in the second row is zero: Then the matrix looks like $A=\begin{pmatrix}
     a & b & c\\
     \al& \be & \ga\\
     x & y & z                                                                                                                              \end{pmatrix}\in G(p,0)$ with only one of $\al,\be, \ga$ is nonzero. Suppose $\al\ne 0,\be=\ga=0.$
     There are $(p-1)$ choices for $\alpha.$ The congruences for $A$ are now : \begin{align*}
     perm(A) &= \alpha(cy + bz) \equiv 0 \pmod{p}, \\ 
     \det(A) &= \alpha(cy - bz) \not\equiv 0 \pmod{p}.
 \end{align*}
 Since $x$ does not appear in the above equations, it is free and has $p$ choices. We need to exclude one choice for $(y,z),$ namely, the unique solution to this system : 
 \begin{align*}
     \alpha(cy + bz)& \equiv 0 \pmod{p},\\ 
     \alpha(cy - bz) &\equiv 0 \pmod{p}.
 \end{align*}
 After throwing that choice of $(y,z),$ we are left with $(p-1)$ choices for the ordered pair $(y,z).$ Thus the total ways in this subcase after including the choices for the first row is $3p(p-1)^5.$ The factor $3$ appears because we could have started with $\beta \neq 0$ or $\gamma \neq 0.$
 \end{itemize}
    
Let $p$ be a prime such that $p-3$ is  not a quadratic residue modulo $p,$ then 
\begin{align*}
 g(p,0) &= 3p^2(p-1)^4 + 3p^2(p-1)^4(p+1) + (p-1)^3[3p(p-1)^2 + 3p(p-1)^3 + p^2(p-1)^2(p-2)] \\ 
 &= 3p^2(p-1)^4 + 3p^2(p-1)^4(p+1) + p(p-1)^5[3 + 3p-3 + p^2 - 2p] \\ 
 &= 3p^2(p-1)^4 + 3p^2(p-1)^4(p+1) + p^2(p-1)^5(p+1) \\ 
  &= p^2(p-1)^4(p^2 + 3p + 5).
\end{align*}
Let's check this with the actual value which a computer code gives out. \begin{enumerate}
    \item When $p = 3, g(3,0) = 9\times16\times23 = 3312$
    \item When $p = 5, g(5,0) = 25\times256\times45 = 288000$
    \item When $p = 7, g(7,0) = 49\times36\times36\times75 = 4762800$
    \item When $p = 11, g(11,0) = 121\times10000\times159 = 192390000.$
    \item When $p = 13, g(13,0) = 169\times144\times144\times213 = 746433792.$
\end{enumerate}
The computer program agrees with our values when $p = 3,5,11$ but not when $p=7,13.$ The code gives an output of $g(7,0) = 4653936$ and $g(13,0) = 739964160.$ So what is different in these cases? Let us consider the case when $p=7$ and the following matrix 
$A = 
\begin{pmatrix}
1 & 1 & 2 \\
1 & 2 & 1 \\
x & y & z
\end{pmatrix}.$ The equations for this matrix become \begin{align*}
    perm(A) &= 5x + 3y + 3z \equiv 0 \pmod{7}, \\
    \det(A) &= 4x + y + z \not\equiv 0 \pmod{7}.
\end{align*}
After a quick inspection, one can see that there is no ordered triad $(x,y,z)$ satisfying the above two congruences as $5(5x + 3y + 3z) \equiv  4x + y + z \pmod{7}.$  But when we take the same matrix but change $p$ to $3,5$ or $11$ then the system admits a solution. So there is something different with the prime $7$ here and $13$ too, because they are throwing out some additional choices for the second row  as well, apart from the multiples of the first row. 

\subsection{Role of Quadratic Residues}
What we noticed that  when $n=7$ or $n=13$  is that each determinant of the three $2\times2$ submatrices formed from last two rows was a non-zero scalar multiple of the respective permanents of those submatrices. Let us explore this further. Note that we are in the case that $p\nmid a\cdot b\cdot c\cdot\al\cdot\be\cdot\ga,$ where $A = \begin{pmatrix}
            a & b & c \\
            \alpha & \beta & \gamma \\
            x & y & z 
            \end{pmatrix} \in G(p,0).$ Since $\ga\ne 0,$ it has $p-1$ choices. 
Suppose that 
\begin{equation}\label{eq:theta}
 \frac{b\gamma + c\beta}{b\gamma - c\beta}=\frac{c\alpha + a\gamma}{c\alpha - a\gamma} = \frac{a\beta + b\alpha}{a\beta - b\alpha} = \theta.
\end{equation}

 It is clear that  $\frac{b\gamma + c\beta}{b\gamma - c\beta}\notin \{1,p-1\}$ if for example, $\frac{b\gamma + c\beta}{b\gamma - c\beta}=1$, then $c\be=0.$ Thus $\theta\in \{2,3,\ldots,p-2\}.$
Also from Equation~\ref{eq:theta} we have the following.
    
$$\alpha = \frac{a\gamma(\theta + 1)}{c(\theta - 1)} = \frac{a\beta(\theta - 1)}{b(\theta + 1)}$$
Again from  $\theta =  \frac{b\gamma + c\beta}{b\gamma - c\beta}$ we have 
$\theta + 1 = \frac{2b\gamma}{b\gamma - c\beta}, \theta - 1 = \frac{2c\beta}{b\gamma - c\beta}.$

Substituting these values, we arrive at 
$ \beta^3 = \frac{b^3\gamma^3}{c^3} = \left(\frac{b\gamma}{c}\right)^3.$

Now from $b\gamma - c\beta \not\equiv 0 \pmod{p}$ we have
   \begin{align*}   
\beta^3 - \left(\frac{b\gamma}{c}\right)^3 &\equiv \left(\beta - \frac{b\gamma}{c}\right)\left(\beta^2 + \left(\frac{b\gamma}{c}\right)\beta + \left(\frac{b\gamma}{c}\right)^2\right) \equiv 0 \pmod{p} \\ 
&\iff \beta^2 + \left(\frac{b\gamma}{c}\right)\beta + \left(\frac{b\gamma}{c}\right)^2 \equiv 0 \pmod{p}.
\end{align*}
This means that $\beta$ has to solve the above quadratic congruence. That can only happen when, the following quadratic congruence in $u$ admits a solution \begin{center}
    $u^2 \equiv \left(\frac{b\gamma}{c}\right)^2 - 4\left(\frac{b\gamma}{c}\right)^2 \equiv -3\left(\frac{b\gamma}{c}\right)^2 \pmod{p}.$
\end{center}
According to Euler's Criterion, the above quadratic congruence in $u$ admits a solution if and only if \begin{center}
    $\left(-3\left(\frac{b\gamma}{c}\right)^2\right)^{\frac{p-1}{2}} \equiv (p-3)^{\frac{p-1}{2}} \equiv 1 \pmod{p} \iff (p-3)$ is a quadratic residue of $p.$ 
\end{center}
 When $(p-3)$ is a quadratic residue of $p,$ then for every choice of $\gamma$ there will be two values of $\beta$ which will each lead to a unique value of $\alpha$ for which we will not get any $(x,y,z)$ satisfying the two congruences. Thus the correction is merely $2(p-1).$ The expression now is $p(p-1)[(p-1)^3 - (p-1) - 2(p-1)] = p(p-1)^2[p^2 + 1 -2p - 1 - 2] = p(p-1)^2(p^2 - 2p - 2).$
Thus, we arrive at the modified expression when $(p-3)$ is a quadratic residue of $p$ as  
    $$g(p,0) = p(p-1)^4[(p+1)^3 + 1].$$

A quick check gives us \begin{enumerate}
    \item $g(7,0) = 7\times36\times36\times513 = 4653936.$ 
    \item $g(13,0) = 13\times144\times144\times2745 = 739964160.$
\end{enumerate}
Both these values agree with the output given by the computer program. Thus Theorem~\ref{thm:th4} is now proved.

\subsection{Computation of $g(p,0,i,j)$}\label{subsec:gpoij}
If $p|x,$ then one can observe  more facts  about the number of elements in the sets $$G(p,x,1,1), G(p,x,1,2), \\ G(p,x,1,3), G(p,x,2,1) \;\mbox{and} \;G(p,x,2,2).$$ Those observations can be derived from the following table.

\begin{tabular}{|c| c| c| c| c| c| c| }
\hline
n  & g(n,0)    & g(n,0,1,1) & g(n,0,1,2) & g(n,0,1,3) & g(n,0,2,1) & g(n,0,2,2)   \\
\hline
3  & 3312      & 2208       & 576        & 96         & 384        & 48              \\
5  & 288000    & 225280     & 38400      & 5120       & 17920      & 1280           \\
7  & 4653936   & 3900960    & 508032     & 54432      & 181440     & 9072          \\
9  & 21730032  & 14486688   & 3779136    & 629856     & 2519424    & 314928       \\
11 & 192390000 & 173140000  & 14520000   & 1100000    & 3520000    & 110000      \\
13 & 739964160 & 677154816  & 49061376   & 3234816    & 10243584   & 269568      \\
\hline
\end{tabular}

\begin{theorem}\label{thm:th5}
Let $p$ be an odd prime. Then 
\begin{enumerate}
 \item \label{thm:th5:1} $g(p,0,2,2) = p(p-1)^4.$ 
 \item \label{thm:th5:2} $g(p,0,2,1) = (3p-1)g(p,0,2,2).$
 \item \label{thm:th5:3} $g(p,0,1,3) = (p-1)g(p,0,2,2).$
 \item \label{thm:th5:4} $g(p,0,1,2) = p(p+1)g(p,0,2,2).$
 \item \label{thm:th5:5} $g(p,0,1,1) = \begin{cases} (p+3)(p^2-p+1)g(p,0,2,2) & \mbox{if } (p-3) \mbox{ is a quadratic residue of  $p$,} \\ 
 (p^3+2p^2+1)g(p,0,2,2) & \mbox{otherwise.}
 \end{cases}$
\end{enumerate}
\begin{proof}
Let $A = \begin{pmatrix}
a_{11} & a_{12} & a_{13} \\
a_{21} & a_{22} & a_{23} \\
a_{31} & a_{32} & a_{33}
\end{pmatrix}.$ Recall the following definitions. 
    $$P_{11} = a_{22}a_{33} + a_{23}a_{32}, \;\;P_{12} = a_{21}a_{33} + a_{23}a_{31},\;\; P_{13} = a_{21}a_{32} + a_{22}a_{31},$$
    $$P_{21} = a_{12}a_{33} + a_{13}a_{32},\;\;\;P_{22} = a_{11}a_{33} + a_{13}a_{31}.$$ Note that in the rest of the proof,  $x=0$ we mean $x\equiv 0 \pmod p.$\\
 Proof of Part~\ref{thm:th5:1}. 
 First we show that that $a_{22}a_{33} - a_{23}a_{32} \not\equiv 0 \pmod{p}.$  For if it was, that would mean $a_{22}a_{33} \equiv 0 \pmod p$ and $a_{23}a_{32} \equiv 0 \pmod p,$ as $p|P_{11}.$ So we are left with either of the four cases $(a_{22},a_{23}) = (0,0)$ or $(a_{22},a_{32}) = (0,0)$ or $(a_{33},a_{23}) = (0,0)$ or $(a_{33},a_{32}) = (0,0).$ Suppose $(a_{22},a_{23}) = (0,0).$ Then from $p|P_{12},$ we have $a_{33} = 0.$ Again from $p|P_{21}$ we have $a_{13} = 0$ or $a_{32} = 0.$ Both these choices contradict invertibility of $A.$ The other three cases can be shown similarly. Thus $a_{22}a_{33} - a_{23}a_{32} \not\equiv 0 \pmod{p}.$ Since 
 $$ a_{21}a_{33} + a_{23}a_{31}\equiv 0\pmod p,\;\;  a_{21}a_{32} + a_{22}a_{31}\equiv 0\pmod p,$$ 
  we are left with $a_{21} = a_{31} = 0.$ Now $\begin{pmatrix}
a_{22} & a_{23} \\
a_{32} & a_{33}
\end{pmatrix}$ has $g_{2}(p,0) = (p-1)^3$ choices and $a_{11}$ has $(p-1)$ choices. Let $a_{12}$ be free to take $p$ choices, after which $a_{13}$ is uniquely determined from $p|P_{21}.$ Thus $g(p,0,2,2) = p(p-1)^4.$

Proof of Part~\ref{thm:th5:2}. 
We claim that any matrix belonging to the $G(p,0,2,1)$ must be in one of the following forms,
$\begin{pmatrix}
a_{11} & a_{12} & a_{13} \\
0 & a_{22} & a_{23} \\
0 & a_{32} & a_{33}
\end{pmatrix}, \begin{pmatrix}
a_{11} & a_{12} & a_{13} \\
a_{21} & 0 & a_{23} \\
a_{31} & 0 & a_{33}
\end{pmatrix}$  and $\begin{pmatrix}
a_{11} & a_{12} & a_{13} \\
a_{21} & a_{22} & 0 \\
a_{31} & a_{32} & 0
\end{pmatrix}.$ If a matrix $A\in G(p,0,2,1)$ is in one of the above three forms, then from $p|perm(A)$ and $p\nmid \det(A)$  the remaining entries of second and third row are nonzero refer~\cite{A:S}. 

First of all we see that if $a_{21} = 0$ then $a_{31}$ must necessarily be $0$ as well. For if it was not the case then from $p \mid P_{13}$ we get $a_{22} = 0.$ Now from $p\mid P_{11}$ we get either $a_{23} = 0$ or $a_{32} = 0.$ The former contradicts invertibility of $A$ while the latter contradicts the congruences.
    $$perm(A) \equiv 0 \pmod{p}, \;\; \det(A) \not\equiv 0 \pmod{p}.$$
So we see that either both $a_{21}$ and $a_{31}$ are $0$, or else none of them is. In a similar way we can prove this for $(a_{22},a_{32})$ and $(a_{23},a_{33}).$ \\
Let $A \in G(p,0,2,1).$ If $a_{21}=0$ then $A$ is of the first form and we are done. If not, then $a_{21} \neq 0$ and $a_{31} \neq 0.$ Now if $a_{22} = 0,$ then $A$ is of the second form and we are done. Otherwise $a_{22}\neq 0$ and $a_{32} \neq 0.$ Thus all the entries of the submatrix $\begin{pmatrix}
a_{21} & a_{22} \\
a_{31} & a_{32}
\end{pmatrix}$ are non-zero. Clearly the determinant of this submatrix cannot be zero, otherwise it would contradict the fact that all the entries are non-zero. Now look at $a_{23}$ and $a_{33}$ as variables in the system of congruences.
$$a_{22}a_{33} + a_{23}a_{32} \equiv 0 \pmod{p}\;\;
P_{12} = a_{21}a_{33} + a_{23}a_{31} \equiv 0 \pmod{p}.$$

Since $a_{21}a_{32} - a_{22}a_{31} \neq 0,$ this system has a unique solution, which is $a_{23} = a_{33} = 0.$ Thus we have shown that every matrix of $G(p,0,2,1)$ must be one of these three forms. We now count each of them.

We count the possible choices when the matrix looks like $\begin{pmatrix}
a_{11} & a_{12} & a_{13} \\
0 & a_{22} & a_{23} \\
0 & a_{32} & a_{33}
\end{pmatrix}.$ There are $g_2(p,0) = (p-1)^3$ ways to choose the submatrix $\begin{pmatrix}
a_{22} & a_{23} \\
a_{32} & a_{33}
\end{pmatrix}$ and $(p-1)$ ways to choose $a_{11}.$ There are $p$ choices for $a_{12}.$ After we choose $a_{12},$ there will be one choice of $a_{13}$ which will make $P_{21} = 0,$ which we do not want. So there are $(p-1)$ choices for $a_{13}.$ So the total choices become $p(p-1)^5.$
We now count the possible choices when the matrix looks like $\begin{pmatrix}
a_{11} & a_{12} & a_{13} \\
a_{21} & 0 & a_{23} \\
a_{31} & 0 & a_{33}
\end{pmatrix}$ or $\begin{pmatrix}
a_{11} & a_{12} & a_{13} \\
a_{21} & a_{22} & 0 \\
a_{31} & a_{32} & 0
\end{pmatrix}.$ In fact, they both have equal number of choices, as will be clear after we illustrate how to count one of them. Suppose we are counting the matrices of the form $\begin{pmatrix}
a_{11} & a_{12} & a_{13} \\
a_{21} & 0 & a_{23} \\
a_{31} & 0 & a_{33}
\end{pmatrix}.$ There are $g_2(p,0) = (p-1)^3$ ways to choose the submatrix $\begin{pmatrix}
a_{21} & a_{23} \\
a_{31} & a_{33}
\end{pmatrix}$ and $(p-1)$ ways to choose $a_{12}.$ Note how $P_{21} = a_{12}a_{33} + a_{13}a_{32} = a_{12}a_{33}.$ So there are $p$ choices each for $a_{13}$ and $a_{11}.$ Thus the total choices for the second form are $p^2(p-1)^4.$ But the other form is dealt with in the exact same way as well. Finally we arrive at 
$$g(p,0,2,1) = p(p-1)^5 + 2p^2(p-1)^4 = p(p-1)^4(3p-1) = (3p-1)g(p,0,2,2).$$

Proof of Part~\ref{thm:th5:3}.
In this case we have $0=perm(A)=a_{13}P_{13}$ and $p\nmid P_{13}$ hence $a_{13} = 0.$   We also claim that $p\nmid a_{21}\cdot a_{22} \cdot a_{23}.$ If $p|a_{21},$ then from $p\mid P_{12}$ we get $a_{23} = 0$ or $a_{31} = 0.$ If $a_{31} = 0,$ then  $p \nmid P_{13}$ gives $a_{21} \neq 0,$ which is absurd. If $a_{23} = 0$ then $p\mid P_{11}$ gives either $a_{22} = 0$ or $a_{33} = 0,$ both of which contradict the invertibility of $A.$ In a similar way we can show that $p\nmid a_{22}a_{23}.$

Now once we fix $a_{23}$ which has $(p-1)$ choices, we see that $a_{33}$ is determined as $a_{33} = \frac{-a_{32}a_{23}}{a_{22}} = \frac{-a_{31}a_{23}}{a_{21}}.$ This gives $a_{21}a_{32} - a_{22}a_{31} = 0.$ So how many choices are there for $B=\begin{pmatrix}
a_{21} & a_{22} \\
a_{31} & a_{32}
\end{pmatrix}$? We claim that none of the $a_{31}$ or $a_{32}$ can be zero. Suppose $a_{31} = 0.$ Then from $p\mid P_{12},$ we have either $a_{21} = 0$  or $a_{33} = 0.$ Thus if $a_{31} = 0$, then we have $a_{33} = 0.$ But then from $a_{33} = 0,$  $p \mid P_{11}$ we get that $a_{32} = 0,$ so $A$ cannot be invertible. It can be proved in a similar way that $a_{32} \ne  0.$ With these observations we are ready to compute $g(p,0,1,3).$

From $a_{21} \neq 0, a_{22} \neq 0$ the total choices for the first row of $B$ are $(p-1)^2$. The second row of $B$ can only be a non-zero multiple of the first row, thus it has $(p-1)$ choices. Thus the total choices for  the matrix $B$ are $(p-1)^3.$ Let $a_{11}$ have $p$ choices. Once we fix $a_{11},$ then it follows from $\det(A) \equiv 0 \pmod{p}$ and $a_{21}a_{33} - a_{23}a_{31} \neq 0$ we have  $(p-1)$ choices for $a_{12}.$ Thus the total choices become $p(p-1)^5 = (p-1)g(p,0,2,2).$\\
Proof of Part~\ref{thm:th5:4}. 
First observe that it is impossible to have $a_{23} = a_{33} = 0$ because that would contradict $p\nmid P_{12}.$
We prove the result in two cases $p|a_{23}\cdot a_{33}$ and $p\nmid a_{23}\cdot a_{33}.$

Let $p|a_{33} .$ Then from $p\mid P_{11} = a_{22}a_{33} + a_{23}a_{32}$ we have $p|a_{32}.$
This leaves $g_2(p,0) = (p-1)^3$ choices for $\begin{pmatrix}
a_{12} & a_{13} \\
a_{22} & a_{23}
\end{pmatrix}$ and $(p-1)$ choices for $a_{31}.$ Finally $a_{11}$ and $a_{12}$ are free to take any values, thus the total choices are $p^2(p-1)^4.$ Same with  $p|a_{23}.$ Thus the total number of matrices when $p|a_{23}\cdot a_{33}$ are $2p^2(p-1)^4.$

We now suppose that $p\nmid a_{23}\cdot a_{33}.$ There are $(p-1)^2$ ways to choose the ordered pair $(a_{23},a_{33}).$ 
In this case it is easy to see that $p\nmid a_{22}a_{32}.$ First observe that $p|P_{11}$ hence $p| a_{22}$ if and only if $p| a_{32}.$ Hence if $p|a_{22},$ then $perm(A)=a_{12}P_{12}$ forces $a_{12}=0$ which contradicts the invertibility of $A.$

 So now $a_{22}$ has $(p-1)$ choices. After fixing $a_{22}, $ $a_{32}$ is uniquely determined from $p\mid P_{11}.$ Now let $a_{21}$ be free to take any of the $p$ values. Since $p \nmid P_{12} = a_{21}a_{33} + a_{23}a_{31},$ there is a unique value of $a_{31}$ which will render $P_{12} = 0.$ We throw out that value, so we are left with $(p-1)$ choices for $a_{31}.$ Now all that is left is to choose the first row.

 Let $a_{13}$ be free to take $p$ values. Then $a_{12}$ is uniquely determined because of the relation $perm(A) =  a_{12}P_{12} + a_{13}P_{13}.$  Finally we need to choose $a_{11}.$ We claim first that $(a_{22}a_{33} - a_{23}a_{32}) \neq 0.$ Otherwise we would end up with multiple zeroes in $\begin{pmatrix}
a_{22} & a_{23} \\
a_{32} & a_{33}
\end{pmatrix}$ which is absurd since in this case all of the four entries are non-zero. Therefore we are left with $(p-1)$ values of $a_{11}$ because we will throw one value out which will make the determinant of $A$ zero. Thus the total choices in this case are $p^2(p-1)^5.$
Thus,   
$$g(p,0,1,2) = 2p^2(p-1)^4 + p^2(p-1)^5 = p^2(p-1)^4(p+1) = p(p+1)g(p,0,2,2).$$

Proof of Part~\ref{thm:th5:5}.
It is easy to see  that $$g(p,0,1,1)=g(p,0)-g(p,0,1,2)-g(p,0,1,3)-g(p,0,2,1)-g(p,0,2,2).$$
\end{proof}
\end{theorem}
We now compute $g(p,0,1,1)$ independently. With that  we have an alternate proof for Theorem~\ref{thm:th4}.

\subsection{Independent proof of Part~\ref{thm:th5:5} of Theorem~\ref{thm:th5}}
Let $A=[a_{ij}]\in G(p,0,1,1).$ Then we have $p\nmid P_{11}.$ Let $D_{11} = \det(B),$ where $B = \begin{pmatrix}
a_{22} & a_{23} \\
a_{32} & a_{33}
\end{pmatrix}.$

Fist we count number of matrices in $G(p,0,1,1)$ with $D_{11}=0.$ First note that none of the entries of $B$ is zero. For example if $a_{22}=0$ then we have  either $a_{23} = 0$ or $a_{32} = 0.$ But that forces $p|P_{11}.$ Hence the matrix $B$ can be choosen $(p-1)^3$ ways, as  $(a_{22},a_{23})$ can be chosen in $(p-1)^2$ ways and then we are left with $(p-1)$ choices  for second row of $B.$ Let $a_{13}, a_{31}$ be free to choose any values, then it is easy to see that each of $a_{12}, a_{21}$ can take $p-1$ values. Thus total number of matrices in $G(p,0,1,1)$ with $D_{11}=0$ are  $p^2(p-1)^5.$

\subsubsection{When $D_{11} \neq 0$}
First of all note that we cannot have $(a_{21},a_{31}) = (0,0)$ for this would contradict either the invertibility of $A$ if  $a_{11} = 0$ or  $perm(A) = 0$ if  $a_{11} \ne 0$ as $p\nmid P_{11}.$ 
The following table illustrates how the rest of the proof follows.\\
\vglue 2mm
\begin{tabular}{|l|l|l|}
\hline
 Condition & Subconditions & Number of matrices\\
 \hline
      \multirow{3}{*}{$p|a_{21} a_{31}$}&& \\
      && $2p(p-1)^4(p^2+1)$\\ 
      &&\\
      \hline
      
\multirow{3}{*}{$p\nmid a_{21} a_{31}$ } &Exactly one of $a_{22}=0$ or $a_{32}=0$& $2p^2(p-1)^5$\\ 
      &$p\nmid a_{22}a_{32}$ but exactly one of $a_{23}$ or $a_{33}$ is zero& $2p(p-1)^6$ \\
      
      & All the entries of 2nd and 3rd row are nonzero &**\\
          \hline
       
  \multicolumn{3}{|l|}
  {$**=\begin{cases}
                     &p(p-1)^4[(p-1)^3-2(p-1)^2] \;\;\mbox{if $p-3$ is  not a quadratic residue modulo $p,$}\\
                    & p(p-1)^4[(p-1)^3-2(p-1)^2-2(p-1)]\;\;\mbox{if $p-3$ is a quadratic residue modulo $p.$}\\
                   \end{cases}$
        }
 \\

      \hline
    \end{tabular}\\
    \vglue 2mm
  Finally   $p^2(p-1)^5 + 2p(p-1)^4(p^2+1) + 2p^2(p-1)^5 + 2p(p-1)^6 + **$ is the required value for $g(p,0,1,1).$
\begin{description}
 \item[$a_{21} = 0$ or $a_{31} = 0$:] Let $a_{21}=0.$ Then $a_{31}$ has $(p-1)$ choices. From the paper~\cite{A:S} we have  $g_2(p,1) = (p-1)(p^2+1).$ Consequently, the total choices for $B$ are  $(p-1)^2(p^2+1)$.  Let $a_{12}$ and $a_{13}$ each be free to take $p$ values. Each time we do this, we get a unique value for $a_{11}$ from $perm(A) = 0.$ So there seem to be $p^2$ possible ordered triads $(a_{11},a_{12},a_{13}).$ However we must remove $p$ of them because those will cause the determinant to become $0.$ Thus the first row has $p^2-p $ choices.  So in this case, the total choices are $2p(p-1)^4(p^2+1).$
\item[When $a_{21} \neq 0$ and $a_{31} \neq 0$ but exactly one of $a_{22} = 0$ or $a_{32} = 0$:]
Suppose $a_{22} = 0.$ Then $a_{33}$ can take $p$ values and $p\nmid a_{23}a_{32}$ as  $P_{11} \neq 0.$  So there are a total of $p(p-1)^2$ choices for the submatrix $B.$ There are $(p-1)^2$ choices for $(a_{21},a_{31}).$ Finally as discussed earlier, there are $p(p-1)$ choices for the first row. Thus the total choices are $2p^2(p-1)^5.$

\item[$a_{21}\neq 0, a_{31}\neq 0, a_{22}\neq0 $ and $a_{32}\neq 0$ but exactly one of $a_{23}$ or $a_{33} = 0$:]
If $a_{23} = 0,$ then we have $(p-1)^5$ choices for the second and third row combined. The first row can be chosen in $p(p-1)$ ways. But we could also start with $a_{33} = 0,$ making the total choices for this case $2p(p-1)^6.$
\item[All the entries in 2nd and 3rd row are nonzero:]
There are $(p-1)^3$ choices for the third row.  It is easy to see that there are $[(p-1)^3 - 2(p-1)^2]$ choices for the second row, when $(p-3)$ is not a quadratic residue of $p.$
This is because we have two more constraints, namely that $P_{11} \neq 0$ and $D_{11} \neq 0.$ So every time we fix $a_{21}$ and $a_{22},$ we will get one value of $a_{23}$ that will make $P_{11} = 0$ and one value will make $D_{11} = 0.$ 
Now  after fixing the third row, we  imitate the proof of part related to the role of quadratic residues in  Theorem~\ref{thm:th4}. So we subtract another $2(p-1)$ from the total choices of the second row whenever $(p-3)$ is a quadratic residue of $p.$ Again, after fixing the second and third row, the first row can be chosen in $p(p-1)$ ways. 
\end{description}
Now from the discussion provided in Section~\ref{sec:computation} we can compute $g(p^k,0,i,j)$ from the identity $g(p^k,0,i,j)=p^{8(k-1)}g(p,0,i,j).$ But note that $g(n,0,i,j)$ is not multiplicative in $n.$ Hence finding the value  of $g(n,0,i,j)$ is an interesting problem.

\end{document}